\documentclass{amsart}

\usepackage{graphicx,mathrsfs,amssymb,amsfonts}

\usepackage[pdfdisplaydoctitle=true,colorlinks=true,urlcolor=blue,citecolor=blue,%
linkcolor=blue,pdfstartview=FitH,pdfpagemode=None,bookmarksnumbered=true]{hyperref}
\usepackage{cmap} 
\usepackage[usenames]{color}
\usepackage{leftidx}
\usepackage[normalem]{ulem}

\newtheorem{theorem}{Theorem}[section]
\newtheorem{lemma}[theorem]{Lemma}
\newtheorem{corollary}[theorem]{Corollary}
\newtheorem{proposition}[theorem]{Proposition}
\newtheorem{remark}[theorem]{Remark}
\newtheorem{definition}[theorem]{Definition}

\numberwithin{equation}{section}

\newcommand{\cz}{{\mathbb C}}
\newcommand{\gz}{{\mathbb Z}}
\newcommand{\nz}{{\mathbb N}}
\newcommand{\rz}{{\mathbb R}}

\newcommand{\scrC}{\mathscr{C}}
\newcommand{\scrD}{\mathscr{D}}
\newcommand{\scrF}{\mathscr{F}}
\newcommand{\scrL}{\mathscr{L}}
\newcommand{\scrR}{\mathscr{R}}
\newcommand{\scrS}{\mathscr{S}}
\newcommand{\ulu}{\underline{u}}
\newcommand{\ulf}{\underline{f}}
\newcommand{\ulU}{\underline{U}}
\newcommand{\ulF}{\underline{F}}
\newcommand{\cl}{\mathrm{cl}}
\newcommand{\const}{\mathrm{const}}
\newcommand{\dbar}{d\hspace*{-0.08em}\bar{}\hspace*{0.1em}}
\newcommand{\eps}{\varepsilon}
\newcommand{\forget}[1]{}
\newcommand{\lra}{\longrightarrow}
\newcommand{\op}{\mathrm{op}}
\newcommand{\spk}[1]{\langle#1\rangle}
\newcommand{\st}{\mbox{\boldmath$\;|\;$\unboldmath}}
\newcommand{\stBig}{\mbox{\boldmath$\;\Big|\;$\unboldmath}}
\newcommand{\tr}{\mathrm{tr}\,}
\newcommand{\wh}{\widehat}
\newcommand{\wt}{\widetilde}

\begin{document}
\title[Maximal $L_p$-regularity of non-local boundary value problems]%
{Maximal $L_p$-regularity of\\ non-local boundary value problems}

\author{R. Denk}
\address{Universit\"at Konstanz, Institut f\"ur Mathematik und Statistik, Konstanz (Germany)}
\email{robert.denk@uni-konstanz.de}

\author{J.\ Seiler}
\address{Universit\`a di Torino, Dipartimento di Matematica, Torino (Italy)}
\email{joerg.seiler@unito.it}

\maketitle

\begin{abstract}
We investigate the $\scrR$-boundedness of operator families belonging to the Boutet de Monvel calculus.
In particular, we show that weakly and strongly parameter-dependent Green operators of nonpositive order are
$\scrR$-bounded. Such operators appear as resolvents of non-local (pseudodifferential) boundary value problems.
As a consequence, we obtain maximal $L_p$-regularity for such boundary value problems.
An example is given by the reduced Stokes equation in waveguides.
\end{abstract}

\section{Introduction}\label{sec:intro}

During the last decade, the theory of maximal $L_p$-regularity turned out to be an important tool in the theory
of nonlinear partial differential equations and boundary value problems. Roughly speaking, maximal regularity in
the sense of well-posedness of the linearized problem is the basis for a fixed-point approach to show (local in time)
unique solvability for the nonlinear problem. Here, the setting of $L_p$-Sobolev spaces with $p\not=2$ is helpful
in treating the nonlinear terms, due to better Sobolev embedding results. Meanwhile, a large number of equations
from mathematical physics has been successfully treated by this method, in particular in fluid dynamics and for free
boundary problems. {Let us mention only Amann} \cite{amann05} for the general concept of
maximal regularity and {Escher, Pr\"uss, Simonett}\cite{escher-pruess-simonett03} for one of the first
applications in fluid mechanics.

A densely defined closed operator $A:\scrD(A)\subset X\to X$ in a Banach space $X$ is said to have maximal
$L_p$-regularity, $1<p<\infty$, in the interval $I=(0,T)$ with $0<T\le\infty$ if   the Cauchy problem
 $${u^\prime(t)+Au(t)}=f(t)\quad (t\in I),\qquad  u(0)=0,$$
has, for any right-hand side $f\in L_p(I,X)$, a unique solution $u$ satisfying
 $$  \|u^\prime\|_{L_p(I,X)}+\|Au\|_{L_p(I,X)}\le C\|f\|_{L_p(I,X)} $$
with a constant $C$ independent of $f$. Here, $W_p^1(I,X)$  refers to the standard $X$-valued first-order
Sobolev space. If $I$ is finite or $A$ is invertible an equivalent formulation is that the map
\begin{equation*}
{ \frac{d}{dt}+A}:  \leftidx{_0}{W}{^1_p}(I,X)\cap L_p(I,\scrD(A)) \lra L_p(I,X)
\end{equation*}
is an isomorphism, where $\leftidx{_0}{W}{^1_p}(I,X)$ denotes the space of all elements in $W_p^1(I,X)$
with vanishing time trace at $t=0$. Note that non-zero initial values can be treated by an application of
related trace theorems. A standard approach to prove maximal regularity is based on operator-valued Mikhlin
type results due to Weis \cite{weis01} and the concept of $\scrR$-boundedness
(see {Denk, Hieber, Pr\"uss} \cite{denk-hieber-pruess03}, Kunstmann, Weis \cite{kunstmann-weis04}).
For a short introduction to  $\scrR$-boundedness, see Section \ref{sec:rbounded} of this paper.

In many applications, the operator $A$ is given as the $L_p$-realization of a differential boundary value problem.
Under appropriate ellipticity and smoothness assumptions, maximal regularity is known to hold in this case
(see, for example, {Denk, Hieber, Pr\"uss} \cite{denk-hieber-pruess03}).
However, several applications demand for generalizations to non-local
(pseudodifferential) operators and boundary value problems. For instance, the Dirichlet-to-Neumann map in a
bounded domain leads to a pseudodifferential operator on the boundary, i.e. on a closed manifold. An example
for a non-local boundary value problem is obtained by the pseudodifferential approach to the Stokes equation
as developed by Grubb and Solonnikov \cite{grubb-solonnikov91} (see also {Grubb} \cite{grubb95} and
{Grubb, Kokholm} \cite{grubb-kokholm93}), which was also one of our motivations.

In the present  {note} we analyze
the $\scrR$-bounded\-ness of operator families belonging to the so-called Boutet de Monvel calculus with
parameter. This is a pseudodifferential calculus containing, in particular, the resolvents to a {large} class of
non-local boundary value problems  {which} allows to describe in great detail the micro-local fine
structure of such resolvents. {A typical} application of the calculus is the following theorem
$($which, in fact, is a simplified version of Theorem 3.2.7 of Grubb \cite{grubb86}$)$:

\begin{theorem}\label{thm:exemplary_thm}
Let $A(\mu)$, $\mu\in\Sigma$ $($an angular subsector of the complex plane$)$, be a parameter-dependent
second order differential operator on a compact manifold $M$ with smooth boundary, and {let} $G(\mu)$ be a
weakly parameter-dependent Green operator of order and type less than or equal to $2$ and regularity at
least $1/2$. Let  $\gamma_0$ and $\gamma_1$ denote Dirichlet and Neumann boundary conditions,
respectively. If the parameter-dependent boundary value problem
\begin{align*}
   \begin{pmatrix}A(\mu)+G(\mu) \\ \gamma_j\end{pmatrix}:\;
    H^{s}_p(M)\lra
   \begin{matrix} H^{s-2}_p(M)\\ \oplus \\ B^{s-j-1/p}_{pp}(\partial M)\end{matrix},\qquad s>1+1/p,
\end{align*}
with $p\in(1,\infty)$ is parameter-elliptic then it is an isomorphism for $|\mu|$ sufficiently large, and
\begin{align}\label{eq:examplary_thm}
   \begin{pmatrix}A(\mu)+G(\mu) \\ \gamma_j\end{pmatrix}^{-1}=
   \begin{pmatrix}P(\mu) & K(\mu)\end{pmatrix},
\end{align}
with $P(\mu)\in B^{-2,0,\nu}(M;\Sigma)$ and a parameter-dependent Poisson operator $K(\mu)$ of order $-j$.
\end{theorem}

The involved operator classes as well as the meaning of parameter-ellipticity will be explained in the sequel;
the mentioned Green operators are certain non-local operators that are smoothing in the interior of $M$, but on
the whole manifold with boundary have a finite order. As a consequence of \eqref{eq:examplary_thm},
 $$A(\mu)+G(\mu):\big\{u\in H^2_p(M)\mid \gamma_j u=0\big\}\subset L_p(M)\lra L_p(M)$$
is invertible for large $\mu$ with inverse $P(\mu)\in B^{-2,0,\nu}(M;\Sigma)$.
Making use of this specific pseudodifferential structure we shall derive, in particular, that
$\{(1+|\mu|)^2P(\mu)\mid \mu\in \Sigma\}\subset\scrL(L_p(M))$ is $\scrR$-bounded, cf. Theorem \ref{thm:main1}.
For the proof we also adopt a tensor-product argument first used in {Denk, Krainer} \cite{denk-krainer07}
in the analysis of the $\scrR$-boundedness of parameter-dependent families of ``scattering" or
``$SG$-pseudodifferential" operators $($which, roughly speaking, allows to reduce considerations to constant
coefficient operators$)$ and use  general results of Kalton, {Kunstmann, Weis} \cite{kalton-kunstmann-weis06}
on the behaviour of $\scrR$-boundedness under interpolation and duality.

There are different versions of Boutet de Monvel's calculus, one with a \emph{strong}, the other with a
\emph{weak} parameter-dependence. The first calculus  is essentially designed to handle fully differential problems,
and is described, for example, in Schrohe, Schulze \cite{schrohe-schulze94}. The second is a broader calculus developed
by Grubb allowing the investigation of certain non-local problems, see Grubb \cite{grubb86} and
Grubb, Kokholm \cite{grubb-kokholm93} for instance. Actually,
we shall blend these two versions and consider operator families depending on two parameters, where one enters in
the strong way and the other only weakly; for details see Section \ref{sec:2}. Though this combination cannot be
found explicitly in the literature, we shall use it freely and avoid giving any proofs, since these are quite standard
(though laborious if done with all necessary details$)$. Our main result is Theorem \ref{thm:main1} stating that such
operator families are $\scrR$-bounded as operator families in the $L_p$-space of the bounded manifold. An
application is provided in Section \ref{sec:stokes} where we consider a resolvent problem for the Stokes operator
in a wave guide (i.e. cylindrical domain) with compact, smoothly bounded cross-section.

Boutet de Monvel's calculus can also be exploited to demonstrate existence of bounded imaginary powers and
even of a bounded $H_\infty$-calculus, cf. Duong \cite{duong90}, Abels \cite{abels05a} and
Coriasco, Schrohe, Seiler \cite{coriasco-schrohe-seiler07} for example; as it turns out, the strategy of proof
we use in the present work is closely related to that of {Abels} \cite{abels05a}.
{On  one hand}, bounded $H_\infty$-calculus is stronger than $\scrR$-boundedness, on the other hand the
concept of $\scrR$-boundedness applies to operator-families more general than the resolvent of a fixed operator.

\section{A short rewiew of $\scrR$-boundedness}\label{sec:rbounded}

We will briefly recall the definition of $\scrR$-boundedness and some results that will be important for our purpose.
For more detailed expositions we refer the reader to {Denk, Hieber, Pr\"uss} \cite{denk-hieber-pruess03} and
{Kunstmann, Weis} \cite{kunstmann-weis04}.
Throughout this section,  {let} $X$, $Y$, $Z$ denote Banach spaces.

A set $T\subset\scrL(X,Y)$ is called $\scrR$-bounded if there exists a $q\in[1,\infty)$ such that
 $$\scrR_q(T):=\sup\Big\{\Big(\!\sum_{z_1,\ldots,z_N=\pm1}\Big\|\sum_{j=1}^N z_jA_jx_j\Big\|^q\Big)^{1/q}
     \Big(\!\sum_{z_1,\ldots,z_N=\pm1}\Big\|\sum_{j=1}^N z_jx_j\Big\|^q\Big)^{-1/q}\Big\}$$
is finite, where the supremum is taken over all $N\in\nz$, $A_j\in T$ and $x_j\in X$ (for which the denominator is
different from zero, of course). The number $\scrR_q(T)$ is called the $\scrR$-bound of $T$. It is a consequence
of Kahane's inequality that finiteness of $\scrR_q(T)$ for a particular $q$ implies finiteness for any
other choice of $q\ge1$. Therefore $q$ is often suppressed from the notation. Clearly an $\scrR$-bounded
set is norm bounded and its norm-bound is majorized by its $\scrR$-bound. In case both $X$ and
$Y$ are {Hilbert spaces}, $\scrR$-boundedness is equivalent to norm-boundedness.

If $S,T\subset\scrL(X,Y)$ and $R\subset\scrL(Y,Z)$ are $\scrR$-bounded then $S+T$ and $RS$ are
$\scrR$-bounded, too,  with
 $$\scrR(S+T)\le\scrR(S)+\scrR(T),\qquad \scrR(RS)\le\scrR(R)\scrR(S).$$

Under mild assumptions on the involved Banach spaces $\scrR$-boundedness {behaves well}  under duality
and interpolation. {The following two results can be found in Kalton-Kunstmann-Weis \cite{kalton-kunstmann-weis06},
Proposition~3.5 and Proposition~3.7, respectively.}

\begin{theorem}\label{thm:dual}
Let $T$ be an $\scrR$-bounded subset of $\scrL(X,Y)$ and assume that $X$ is
$B$-convex\footnote{For a definition of $B$-convexity we refer the reader to \cite{kalton-kunstmann-weis06}.
For us it will be sufficient to know that $L_p$-spaces with $1<p<\infty$ are $B$-convex.}.
Then
 $$T^\prime:=\big\{A^\prime\st A\in T\big\}\qquad (\text{set of dual operators})$$
is an $\scrR$-bounded subset of $\scrL(Y^\prime,X^\prime)$ with $\scrR(T^\prime)\le C \scrR(T^\prime)$ with
a constant $C\ge 0$ not depending on $T$.
\end{theorem}

\begin{theorem}\label{thm:inter}
Let $(X_0,X_1)$ and $(Y_0,Y_1)$ be two interpolation couples with both $X_0$ and $X_1$ being $B$-convex.
Let $T\subset\scrL(X_0+X_1,Y_0+Y_1)$ such that
$T\subset\scrL(X_j,Y_j)$ is $\scrR$-bounded with $\scrR$-bound $\kappa_j$ for $j=0,1$. Then
 $$T\subset\scrL\big((X_0,X_1)_{\theta,p},(Y_0,Y_1)_{\theta,p}\big),\qquad 0<\theta<1,\quad 1<p<\infty,$$
is $\scrR$-bounded with $\scrR$-bound $\kappa\le \kappa_0^{1-\theta} \kappa_1^{\theta}$, where
$(\cdot,\cdot)_{\theta,p}$ refers to the real interpolation method.
\end{theorem}

\forget{
\begin{proposition}\label{prop:mult}
Let $\Omega\subset\rz^n$ be open and $1\le p<\infty$. For $\varphi\in L_\infty(\Omega)$ let
$M_\varphi\in\scrL\big(L_p(\Omega,X)\big)$ denote the operator of multiplication by $\varphi$.
If $\Phi\subset L_\infty(\Omega)$ is a bounded set then
 $$T_\Phi:=\{M_\varphi\st\varphi\in\Phi\}\subset \scrL\big(L_p(\Omega,X)\big)$$
is $\scrR$-bounded with
 $$\scrR_p(T_\Phi)\le 2\sup_{\varphi\in\Phi}\|\varphi\|_\infty.$$
\end{proposition}
}

The following {statement} (Proposition~3.3 in {Denk, Hieber, Pr\"uss} \cite{denk-hieber-pruess03}) is very useful in analyzing the
$\scrR$-boundedness of families of integral operators.

\begin{theorem}\label{thm:intop}
Let $\Omega\subset\rz^n$ be open, $1< p<\infty$, and assume that
 $$(K_0 f)(\omega)=\int_\Omega k_0(\omega,\omega^\prime)f(\omega)\,d\omega$$
defines an integral operator $K_0\in\scrL(L_p(\Omega))$.
Let $\big\{k_\lambda:\Omega\times\Omega\to\scrL(X,Y)\st\lambda\in\Lambda\big\}$ be a family of measurable
integral kernels and
$T=\big\{K_\lambda\st\lambda\in\Lambda\big\}$ be the set of associated integral operators. If
 $$\scrR_p\Big(\big\{k_\lambda(\omega,\omega^\prime)\st\lambda\in\Lambda\big\}\Big)\le
     k_0(\omega,\omega^\prime)\qquad\text{for all }\omega,\omega^\prime\in\Omega$$
then $T\subset \scrL\big(L_p(\Omega,X),L_p(\Omega,Y)\big)$ is $\scrR$-bounded with
 $$\scrR_p\Big(\big\{K_\lambda\st\lambda\in\Lambda\big\}\Big)\le \|K_0\|_{\scrL(L_p(\Omega))}.$$
\end{theorem}

Now consider $a:\rz^\ell\times\rz^\ell\to\scrL(X,Y)$ and let
 $$[\op(a)u](y)=\int e^{iy\eta}a(y,\eta)\wh{u}(\eta)\,\dbar\eta,\qquad u\in\scrS(\rz^\ell,X),$$
denote the associated pseudodifferential operator, where $\dbar\eta = (2\pi)^{-\ell/2}d\eta$.
Under suitable assumptions on $a$ we have
$\op(a):\scrS(\rz^\ell,X)\to L_p(\rz^\ell,Y)$, say. One may ask when this operator induces a continuous map
$L_p(\rz^\ell,X)\to L_p(\rz^\ell,Y)$. In answering this question the concept of $\scrR$-boundedness plays a
{decisive} role. For example, Girardi and Weis \cite{girardi-weis03} have shown the following:

\begin{theorem}\label{thm:girardi}
Let both $X$ and $Y$ have properties $(\mathcal{H}\mathcal{T})$ and
$(\alpha)$.\footnote{For the definition of these properties we refer the reader to \cite{kunstmann-weis04} or
\cite{denk-hieber-pruess03}.
For us it is sufficient to know that scalar-valued $L_p$-spaces, $1<p<\infty$, have these properties.}
Let $T\subset\scrL(X,Y)$ be $\scrR$-bounded. Then
\begin{align*}
 \Big\{op(a)\st & a\in\scrC^\ell(\rz^\ell_\eta\setminus\{0\},\scrL(X,Y))\text{ with } \\
                        & \eta^\alpha D^\alpha_\eta a(\eta)\in T\text{ for all $\eta\not=0$ and $\alpha\in\{0,1\}^\ell$}\Big\}
\end{align*}
is an $\scrR$-bounded subset of $\scrL\big(L_p(\rz^\ell,X),L_p(\rz^\ell,Y)\big)$ with $\scrR$-bound less than or equal to
$C\scrR(T)$ for some constant $C$ not depending on $T$.
\end{theorem}

In other words, this Theorem of Girardi and Weis is the operator-valued generalization of the classical theorem of
Lizorkin on the continuity of Fourier multipliers in $L_p$-spaces. As an immediate consequence one obtains:

\begin{corollary}\label{cor:S-scrR}
Denote by $S^d_\scrR(\rz^{\ell};X,Y)$, $d\in\rz$, the space of all smooth functions $a:\rz^\ell_\eta\to\scrL(X,Y)$
such that
$T_{\alpha}(a):=\big\{\spk{\eta}^{-d+|\alpha|}D^\alpha_\eta a(\eta)\st \eta\in \rz^{\ell}\big\}$
is an $\scrR$-bounded subset of $\scrL(X,Y)$ for any choice of the multi-index $\alpha$.
As shown in {Denk, Krainer} \textnormal{\cite{denk-krainer07}},
this is a Fr\'{e}chet space, by taking as semi-norms the $\scrR$-bounds of $T_{\alpha}(a)$.
If both $X$ and $Y$ have properties $(\mathcal{H}\mathcal{T})$ and $(\alpha)$ then $\op$ induces
a continuous mapping
 $$S^0_\scrR(\rz^{\ell};X,Y)\lra \scrL\big(L_p(\rz^\ell,X),L_p(\rz^\ell,Y)\big).$$
\end{corollary}

For the interested reader, we refer to  {Portal, Strkalj} \cite{portal-strkalj06} for a more general result on the $L_p$-continuity
of pseudodifferential operators with symbols in {operator-valued} $S^0_{\varrho,\delta}$-classes of H\"ormander type.

\section{Boutet de Monvel's calculus with parameters}\label{sec:2}

In this section, we will present some elements of a parameter-dependent version of Boutet de Monvel's calculus
\cite{boutetdemonvel71} which we use to describe solution operators of parameter-elliptic boundary value problems
subject to homogeneous boundary conditions. The elements of this calculus  are operators of the form
\begin{equation}\label{eq:boutet}
 P(\tau,\mu) = A_{+}(\tau,\mu)+G(\tau,\mu):\;\scrS(\rz^n_+)\lra\scrS(\rz^n_+)
\end{equation}
(extending by continuity to Sobolev spaces), where $k,\ell\in\nz$ are some natural {numbers},
$\rz^n_+$ denotes the half-space
 $$\rz^n_+=\Big\{x=(x^\prime,x_n)\in\rz^{n}\st x_n>0\Big\}$$
and $\scrS(\rz^n_+)$ consists of all functions obtained by restricting rapidly decreasing functions from $\rz^n$
to the half-space $\rz^n_+$ $($this space is a Fr\'{e}chet space by identification with the quotient space
$\scrS(\rz^n)/N$, where $N:=\{u\in\scrS(\rz^n_+)\mid u=0\text{ on }\rz^n_+\}$ is a closed subspace of
$\scrS(\rz^n)$$)$.

In \eqref{eq:boutet}, $A_+(\tau,\mu)$ is a parameter-dependent pseudodifferential operator and $G(\tau,\mu)$ a
so-called parameter-dependent Green operator (one also speaks of \emph{singular} Green operators; however for
convenience we omit the term `singular'). We shall consider two classes of Green operators which are \emph{weakly}
and \emph{strongly} parameter-dependent, respectively.

In the following, we let $\Sigma$ denote a closed sector in the two-dimensional plane
with vertex at the origin. We call a function smooth on $\Sigma$ provided all partial
derivatives exist in the interior and extend continuously to  $\Sigma$.

We shall frequently make use of pseudodifferential symbols taking values in Fr\'{e}chet spaces.
To this end, let us give the following definition:

\begin{definition}\label{def:E-symb}
Let $E$ be a Fr\'{e}chet space with a system $\{p_j\st j\in\nz\}$ of semi-norms determining
its topology. We let $S^{d}(\rz^{m}; E)$, $d\in\rz$, denote the space of all smooth functions
$a:\rz^{m}\to E$ satisfying uniform estimates
\begin{equation}\label{eq:E-valued}
 q_{j,\alpha}(a):=\sup_{y\in\rz^m}p_j\big(\spk{y}^{|\alpha|-d}D^\alpha_{y} a(y)\big)<\infty
\end{equation}
for every $j$ and every multi-index $\alpha$. These semi-norms make $S^{d}(\rz^{m}; E)$ a Fr\'{e}chet
space. In case $E=\cz$ we suppress $E$ from the notation.

The subspace $S^{d}_{\cl}(\rz^{m}; E)$ consists, by definition, of those symbols that
have an expansion into homogeneous components in the following sense: there exist
$a_{(d-\ell)}\in\scrC^\infty(\rz^{m}\setminus\{0\},E)$ satisfying
 $$a_{(d-\ell)}(ty)=t^{d-\ell}a_{(d-\ell)}(y),\qquad t>0,\quad y\not=0, $$
such that
  $$R_N(a)(y):=a(y)-\sum_{\ell=0}^{N-1}\chi(y)a_{(d-\ell)}(y)\;\in\;
      S^{d-N}(\rz^{m}; E)$$
for any $N\in\nz$, where $\chi$ denotes an arbitrary zero-excision function.
\end{definition}

The space of smooth positively homogeneous functions
$\rz^{m}\setminus\{0\}\to E$
of a fixed degree is canonically isomorphic to $\scrC^\infty({\mathbb S}^{m-1},E)$, the
smooth $E$-valued functions on the unit-sphere in $\rz^m$. We then equip
$S^{d}_{\cl}(\rz^{m}; E)$ with the projective topology with respect to the maps
\begin{align*}
\begin{split}
 a\mapsto a_{(d-\ell)}&:\; S^{d}_{\cl}(\rz^{m}; E)\lra \scrC^\infty({\mathbb S}^{m-1},E),\\
 a\mapsto R_N(a)&:\; S^{d}_{\cl}(\rz^{m}; E)\lra S^{d-N}(\rz^{m}; E),
\end{split}
\end{align*}
where $N$ and $\ell$ run through the non-negative integers.
It will be of some importance for us that
\begin{equation}\label{eq:tensor}
  S^d_\cl(\rz^m;E)=S^d_\cl(\rz^m)\,\wh{\otimes}_\pi\,E,
\end{equation}
where $F\,\wh{\otimes}_\pi\,E$ denotes the completed projective tensor-product of the two
Fr\'{e}chet spaces $E$ and $F$, see for example {Tr\`eves}~\cite{treves67}. In other words, $S^d_\cl(\rz^m;E)$
can be identified with the closure of the algebraic tensor product
 $$S^d_\cl(\rz^m)\,\otimes\,E=\Big\{\sum_{i=1}^N a_ie_i
   \stBig N\in\nz,\; a_i\in S^d_\cl(\rz^m),\; e_i\in E \Big\}$$
with respect to the system of semi-norms
 $$\wh{q}_{j,\alpha}(a)=\inf\Big\{\sum_{i=1}^N q_\alpha(a_i)p_j(e_i)
   \stBig a=\sum_{i=1}^N a_ie_i\Big\},$$
where $q_\alpha$ is as in \eqref{eq:E-valued} with $E=\cz$.

\subsection{Parameter-dependent pseudodifferential operators}\label{sec:2.1}

Let us denote by
\begin{equation*}
 S^{d}_{\const}(\rz^n\times\rz\times\Sigma),\qquad d\in\rz,
\end{equation*}
the space of all smooth functions $a:\rz^{n}_\xi\times\rz_\tau\times\Sigma_\mu\lra\cz$ satisfying
 $$\sup_{(\xi,\tau,\mu)\in\rz^{n}\times\rz\times\Sigma}\big|D^\alpha_\xi D^k_\tau
   D^\gamma_\mu a(\xi,\tau,\mu)\big|
   \spk{\xi,\tau,\mu}^{|\alpha|+|\gamma|+k-d}<\infty$$
for every order of derivatives. This is a Fr\'{e}chet space and we can define
\begin{equation}\label{eq:Sd}
 S^{d}(\rz^n\times\rz^n\times\rz\times\Sigma):=
 S^0_\cl(\rz^n_x;{S^{d}_{\const}(\rz^n_\xi\times\rz_\tau\times\Sigma_\mu)}).
\end{equation}

{
\begin{remark}\label{rem:12345}
Let us emphasise that a symbol $a(x,\xi,\tau,\mu)$ from \eqref{eq:Sd}
{does not only satisfy} the {standard} uniform symbol estimates
 $$\sup_{(x,\xi,\tau,\mu)\in\rz^n\times\rz^{n}\times\rz\times\Sigma}
     \big|D^\beta_x D^\alpha_\xi D^k_\tau D^\gamma_\mu a(x,\xi,\tau,\mu)\big|
     \spk{\xi,\tau,\mu}^{|\alpha|+|\gamma|+k-d}<\infty$$
but also has an expansion into homogeneous components {of decreasing degree} with
respect to the $x$-variable. In particular, if $a$ satisfies the above estimates and has compact
$x$-support, it belongs to {the space  \eqref{eq:Sd}}. Symbols of the latter type typically
arise when working on compact manifolds, by using local coordinate systems and subordinate
partitions of unity.
\end{remark}
}

With a symbol $a$ from \eqref{eq:Sd} we associate a family of pseudodifferential operators
 $$A(\mu,\tau)=\op(a)(\mu,\tau):\scrS(\rz^n)\to\scrS(\rz^n)$$
in the standard way, i.e.,
 $$[A(\mu,\tau)u](x)=\int e^{ix\xi}a(x,\xi,\tau,\mu)\widehat{u}(\xi)\,\dbar\xi.$$
This map can be extended to a map $\scrS^\prime(\rz^n)\to\scrS^\prime(\rz^n)$ in the space of tempered
distributions.
Now let
 $$e_+:\scrS(\rz^n_+)\lra \scrS^\prime(\rz^n), \qquad r_+:\scrS^\prime(\rz^n)\lra \scrD^\prime(\rz_+^n),$$
be the operators of extension by $0$ and restriction to the half-space, respectively.
For $A(\mu,\tau)$ as above we set
\begin{equation*}
 A_+(\mu,\tau)=\op(a)_+(\tau,\mu):=r_+\circ A(\mu,\tau)\circ e_+.
\end{equation*}
This gives rise to a map $\scrS(\rz^n_+)\to \scrC^\infty(\rz^n_+)$, for example.
If $d=0$ it induces maps
\begin{equation}\label{eq:op_+}
 A_+(\mu,\tau):L_p(\rz^n_+) \lra L_p(\rz^n_+),\qquad 1<p<\infty.
\end{equation}
It is this mapping \eqref{eq:op_+} we will be most interested in, and we shall analyze it below for the symbol class
we have just introduced. However, for motivations of the calculus $($for example, to ensure that $A_+(\tau,\mu)$
preserves the space $\scrS(\rz^n_+)$ and that the operators behave nicely under standard operations like
composition$)$ one actually needs to require an additional property of the symbols: the so-called two-sided
\emph{transmission condition} with respect to the boundary of $\rz^n_+$. For a symbol $a$ of order $d$ as above,
the condition requires that, for any choice of $k\in\nz_0$,
\begin{align*}
 \scrF^{-1}_{\xi_n\to z}D^k_{x_n}&
     p(x^\prime,0,\xi^\prime,\spk{\xi^\prime,\tau,\mu}\xi_n,\tau,\mu)\Big|_{\pm z>0}\\
     &\in{ S^{-k}_\cl}\Big(\rz^{n-1}_{x^\prime};{S^d\big(\rz^{n-1}_{\xi^\prime}\times\rz_\tau\times\Sigma_\mu;
     \scrS({\rz}_{\pm,z})}\big)\Big),
\end{align*}
i.e., the  restriction of the distribution
$\scrF^{-1}_{\xi_n\to z}D^k_{x_n}
p(x^\prime,0,\xi^\prime,\spk{\xi^\prime,\tau,\mu}\xi_n,\tau,\mu)\in\scrS^\prime(\rz_z)$
to the half-space $\rz_+$ or $\rz_-$ can be extended to a rapidly decreasing function on $\rz$, and the other
variables enter as parameters in the indicated specific way
{(cf.\ Definition 2.3 in {Schrohe} \cite{schrohe01}, replacing there $\xi^\prime$ by $(\xi^\prime,\tau,\mu)$ and
passing to the inverse Fourier transform; note that the inverse Fourier transform of a polynomial has support in the
origin $z=0$ and thus is eliminated by restriction to the half-line $\pm z>0$).
}
Here, $\spk{\xi',\tau,\mu} := (1+|\xi'|^2+\tau^2+|\mu|^2)^{1/2}$.
Symbols with the transmisssion condition form a closed subspace of $S^{d}(\rz^n\times\rz^n\times\rz\times\Sigma)$
that we shall denote by
\begin{equation*}
 S^{d}_\tr(\rz^n\times\rz^n\times\rz\times\Sigma).
\end{equation*}

\begin{remark}
The operator $A_+(\mu,\tau)=\op(a)_+(\tau,\mu)$ does not depend on the values of the symbol
$a$ for $x_n<0$. Hence, if we define $S^{d}_{-}(\rz^n\times\rz^n\times\rz\times\Sigma)$
as the closed subspace of symbols from $S^{d}_\tr(\rz^n\times\rz^n\times\rz\times\Sigma)$
whose $x$-support is contained in half-space $\{x\in\rz^n\st x_n\le0\}$, then the class of
operators is isomorphic to the quotient
 $$S^{d}_\tr(\rz^n\times\rz^n\times\rz\times\Sigma)\big/
   S^{d}_-(\rz^n\times\rz^n\times\rz\times\Sigma),$$
yielding a natural Fr\'{e}chet topology.
\end{remark}

\subsection{Parameter-dependent Green operators}\label{sec:2.2}

We shall use the splitting $\rz_+^n=\rz^{n-1}\times\rz_+$ and write $x=(x^\prime,x_n)$ and, correspondingly,
$\xi=(\xi^\prime,\xi_n)$ for the covariable $\xi$ to $x$. Roughly speaking, Green operators in tangential direction
(i.e., on $\rz^{n-1}$) act like pseudodifferential operators while in normal direction (i.e., on $\rz_+$) they act like
integral operators with smooth kernel. However, there is a certain twisting between the two directions which reflects
in a specific structure of the operators. To describe this structure we shall need the function $\varrho$ defined by
\begin{equation*}
 \varrho(\xi^\prime,\tau,\mu):=\spk{\mu}\spk{\xi^\prime,\tau,\mu}^{-1}.
\end{equation*}
Note that $0<\varrho\le1$. Now let
\begin{equation*}
 R^{d,\nu}_\const(\rz^{n-1}\times\rz\times\Sigma),\qquad {d\in\rz,\;\nu\ge0},
\end{equation*}
denote the space of all smooth scalar-valued functions $k(\xi^\prime,\tau,\mu;x_n,y_n)$
satisfying uniform estimates
\begin{align}\label{eq:kernel5}
\begin{split}
   \big\|x_n^\ell D^{\ell^\prime}_{x_n}
 & y_n^m D^{m^\prime}_{y_n}
   D^{\alpha^\prime}_{\xi^\prime}D^k_\tau D_\mu^{\gamma}
   k(\xi^\prime,\tau,\mu;x_n,y_n)
   \big\|_{L^2(\rz_{+,x_n}\times\rz_{+,y_n})}\\
 \le&
   C_{\alpha^\prime,\ell,\ell^\prime,m,m^\prime}(k)
   \Big(
    \varrho(\xi^\prime,\tau,\mu)^{\nu-[\ell-\ell^\prime]_+-[m-m^\prime]_+ -|\alpha^\prime|-k}+1
   \Big)\times\\
 &\times\spk{\xi^\prime,\tau,\mu}^{
   d-\ell+\ell^\prime-m+m^\prime-|\alpha^\prime|-k-|\gamma|},
\end{split}
\end{align}
for every order of derivatives and any $\ell,m\in\nz_0$; here $[s]_+=\max(s,0)$ for any real number $s$.
We call such a $k$ a \emph{weakly parameter-dependent} symbol kernel of order $d$ and regularity $\nu$
$($with constant coefficients$)$, see also \cite{grubb86}. The best constants define a system of semi-norms,
yielding a Fr\'{e}chet topology. We set
\begin{equation*}
 R^{d,\nu}(\rz^{n-1}\times\rz^{n-1}\times\rz\times\Sigma)=
 S^0_\cl(\rz^{n-1}_{x^\prime};{R^{d,\nu}_\const(\rz^{n-1}_{\xi^\prime}\times\rz_\tau\times\Sigma_\mu)})
\end{equation*}
The class of \emph{strongly parameter-dependent} symbol kernels
\begin{equation}\label{eq:rdconst}
 R^{d}_\const(\rz^{n-1}\times\rz\times\Sigma):=\bigcap_{\nu\in\rz}
 R^{d,\nu}_\const(\rz^{n-1}\times\rz\times\Sigma)
\end{equation}
consists of those symbol kernels satisfying the uniform estimates
\begin{align*}
\begin{split}
   \big\|x_n^\ell D^{\ell^\prime}_{x_n}
 & y_n^m D^{m^\prime}_{y_n} D^\alpha_{\xi^\prime} D^k_\tau D^\gamma_\mu
   k(\xi^\prime,\tau,\mu;x_n,y_n)\big\|_{L_2(\rz_{+,x_n}\times\rz_{+,y_n})} \\
 & \le C_{\alpha^\prime,\ell,\ell^\prime,m,m^\prime}(k)
   \spk{\xi^\prime,\tau,\mu}^{d-\ell+\ell^\prime-m+m^\prime-|\alpha|-|\gamma|-k}
\end{split}
\end{align*}
for any order of derivatives and any $\ell,m\in\nz_0$.
Again this is a Fr\'{e}chet space, and we have
\begin{align}\label{eq:rd}
\begin{split}
   R^{d}(\rz^{n-1}\times\rz^{n-1}\times\rz\times\Sigma):=&
   \mathop{\mbox{\Large$\cap$}}_{\nu\in\rz}
   R^{d,\nu}(\rz^{n-1}\times\rz^{n-1}\times\rz\times\Sigma)\\
 =& S^0_\cl(\rz^{n-1}_{x^\prime};{R^{d}_\const(\rz^{n-1}_{\xi^\prime}\times\rz_\tau\times\Sigma_\mu)}).
\end{split}
\end{align}

{Let us point out once more the dependence on $x^\prime$ as a classical symbol of order $0$ and not only as
a function bounded with all its derivatives, cf.\ Remark \ref{rem:12345}. In particular, the class of Green symbols
defined above is a subclass of that defined in \cite{grubb-kokholm93}.}

\begin{definition}\label{def:green1}
A weakly parameter-dependent Green operator of order
$d\in\rz$, \emph{type} $r=0$, and regularity $\nu$ is of the form
\begin{equation}\label{eq:kernel4}
 [G(\tau,\mu)u](x)=\int e^{ix^\prime\xi^\prime} \int_0^\infty
 k(x^\prime,\xi^\prime,\tau,\mu;x_n,y_n)\scrF_{y^\prime\to\xi^\prime}
 u(\xi^\prime,y_n)\,dy_n\dbar\xi^\prime
\end{equation}
where $k\in R^{d,\nu}(\rz^{n-1}\times\rz^{n-1}\times\rz\times\Sigma)$ is a weakly
parameter-dependent symbol kernel of order $d$ and regularity $\nu$ as introduced above; we
occasionally shall write $G(\tau,\mu)=\op(k)(\tau,\mu)$.
Parameter-dependent Green operators of order $d\in\rz$, type $r\in\nz$, and regularity $\nu$ have
the form
\begin{equation}\label{eq:sum}
 G(\tau,\mu)=G_0(\tau,\mu)+\sum_{j=1}^{r} G_j(\tau,\mu)D_{x_n}^j
\end{equation}
where each $G_j$ has order $d-j$, type $0$, and regularity $\nu$.
We shall denote this class of operators by $G^{d,r,\nu}(\rz^{n}_+;\rz\times\Sigma)$.
Analogously, we obtain the classes $G^{d,r}(\rz^{n}_+;\rz\times\Sigma)$ of strongly
parameter-dependent Green operators, using strong\-ly parameter-dependent symbols kernel.
The subclasses $G^{d,r,\nu}_\const$ and $G^{d,r}_\const$ refer to symbol kernels
that do not depend on the $x^\prime$-variable.
\end{definition}

All the previously introduced spaces of Green operators inherit a Fr\'{e}chet topology from the underlying spaces of
symbol kernels (factoring out the ambiguity of representing Green operators as different linear combinations
{or, in other words, forming the non-direct sum of  Fr\'{e}chet spaces}).

{
The definition of Green operators in Definition \ref{def:green1} is in the spirit of {Schrohe, Schulze} \cite{schrohe-schulze94} and, at
the first glance, differs from the original approach in {Boutet de Monvel} \cite{boutetdemonvel71}, cf.\ also {Grubb, Kokholm} \cite{grubb-kokholm93}.
However, both approaches are equivalent as was shown in Lemma 2.2.14 of {Schrohe, Schulze} \cite{schrohe-schulze94} in case of
strong parameter-dependence; weak parameter-dependence can be treated similarly.
}

{
Below we shall make use of an alternative characterisation of strongly parameter-dependent
Green operators, see Theorem 3.7 in {Schrohe} \cite{schrohe01}, for instance (our variables $x^\prime$ and
$(\xi^\prime,\tau,\mu)$ {correspond} to $y$ and $\eta$, respectively, in \cite{schrohe01};
the assumption  in \cite{schrohe01} that both $y$ and $\eta$ belong to some $\rz^q$ is only devoted to the
context and can be relaxed without any difficulty to $y\in\rz^p$ and $\eta\in\rz^q$ with different dimensions
$p$ and $q$).
}

\begin{proposition}\label{prop:kernel}
Any strongly parameter-dependent Green operator of order $d$ and type $0$ has a symbol kernel of the form
\begin{equation*}
 k(x^\prime,\xi^\prime,\tau,\mu;x_n,y_n)=\wt{k}(x^\prime,\xi^\prime,\tau,\mu;
 \spk{\xi^\prime,\tau,\mu}x_n,\spk{\xi^\prime,\tau,\mu}y_n).
\end{equation*}
Here,
\begin{equation*}
 \wt{k}(x^\prime,\xi^\prime,\tau,\mu;s_n,t_n)\,\in\,
 S^0_{\cl}\big(\rz^{n-1}_{x^\prime};
 S^{d+1}\big({\rz^{n-1}_{\xi^\prime}\times\rz_\tau\times\Sigma_\mu};
 \scrS({\rz}_{+,s_n}\times{\rz}_{+,t_n})\big)\big),
\end{equation*}
where $\scrS({\rz}_+\times{\rz}_+)=\scrS(\rz^2)\big|_{\rz_+\times\rz_+}$ and
$S^{d}(\rz^{n-1}\times\rz\times\Sigma;E)$ for a Fr\'{e}chet space $E$ is defined as in Definition
$\mathrm{\ref{def:E-symb}}$,
replacing $\rz^m$ by $\rz^{n-1}\times\rz\times\Sigma$.
\end{proposition}

\subsection{Some elements of the calculus}\label{sec:2.3}

Having described parameter-dependent pseudodifferential and Green operators let us introduce the spaces
 $$ B^{d,r}_{(\mathrm{const})}(\rz^{n}_+;\rz\times\Sigma),
    \qquad
    B^{d,r,\nu}_{(\mathrm{const})}(\rz^{n}_+;\rz\times\Sigma) $$
consisting of operators $A_+(\tau,\mu)+G(\tau,\mu)$ with a parameter-dependent pseudodifferential operator
of order $d\in\gz$ as in Section \ref{sec:2.1} and a -- strongly or weakly -- parameter-dependent Green operator
of order $d$, type $r\in\nz_0$, and regularity $\nu\ge0$ as described in Section \ref{sec:2.2}. Using the topologies
of both pseudodifferential operators and Green operators introduced above we obtain natural topologies as
non-direct sums of Fr\'{e}chet spaces. Considering the parameter-dependent operators as families of operators
$\scrS(\rz^n_+)\to\scrS(\rz^n_+)$, the following results hold
({Theorem~5.1 and Theorem~5.3, respectively,  in Grubb, Kokholm \cite{grubb-kokholm93}}).

\begin{theorem}\label{thm:algebra}
The pointwise composition of operators induces continuous mappings
\begin{align*}
 B^{d_0,r_0,\nu_0}(\rz^n_+;\rz\times\Sigma)\times
 B^{d_1,r_1,\nu_1}(\rz^n_+;\rz\times\Sigma)
 \lra B^{d,r,\nu}(\rz^n_+;\rz\times\Sigma)
\end{align*}
where
 $$d=d_0+d_1,\quad r=\max\{r_1,r_0+d_1\},\quad
   \nu=\min\{\nu_0,\nu_1\}.$$
Moreover, the subclass of Green operators forms an ideal, i.e., is preserved under composition from the left or
the right by operators of the full class. Similar statements hold for the classes of strongly parameter-dependent
operators {$($by passing to the intersection over all regularities $\nu\ge0)$}.
\end{theorem}

\begin{theorem}\label{thm:adjoint}
If $d\le0$, taking the $($formal$)$ adjoint with respect to the $L_2(\rz^n_+)$-inner products induces
continuous mappings
\begin{align*}
 B^{d,0;\nu}(\rz^n_+;\rz\times\Sigma)
   &\lra B^{d,0;\nu}(\rz^n_+;\rz\times\Sigma).
\end{align*}
The subclasses of Green operators are preserved under taking adjoints.
{Similar statements hold for the classes of strongly parameter-dependent operators.}
\end{theorem}

It has been shown in {Grubb, Kokholm}  \cite{grubb-kokholm93} that the operators extend by continuity from the spaces of
Schwarz functions to $L_p$-Sobolev spaces. In fact, if we set, with $s\in\rz$ and $1<p<\infty$,
\begin{align*}
 H^s_p(\rz^n_+)
  &=\Big\{u|_{\rz^n_+} :  u\in H^s_p(\rz^n) \Big\}\cong H^s_p(\rz^n)\big/N^s_{p},
\end{align*}
where
 $$N^s_{p}:=\Big\{u\in H^s_p(\rz^n) \st \mathrm{supp}\,u\subset \rz^{n-1}\times(-\infty,0]\Big\},$$
 then any element of $B^{d,r,\nu}(\rz^n_+;\rz\times\Sigma)$ induces
 pointwise $($i.e., for any fixed value of $(\tau,\mu))$  mappings
\begin{equation}\label{eq:Sobmap}
  H^s_{p}(\rz^n_+)\lra H^{s-d}_p(\rz^n_+),\qquad s>r-1+\frac{1}{p}.
\end{equation}
If we let $B^{d,r}(\rz^n_+)$ be the Fr\'echet space of operators not depending on the parameters $\tau,\mu$
$($which is obtained as above by eliminating everywhere the parameters$)$, we have
 $$B^{d,r,\nu}(\rz^n_+;\rz\times\Sigma)\hookrightarrow
     C^\infty(\rz\times\Sigma)\,\wh{\otimes}_\pi\,B^{d,r}(\rz^n_+);$$
in particular, whenever $s>r-1+\frac{1}{p}$,
\begin{equation}\label{eq:smooth123}
 B^{d,r,\nu}(\rz^n_+;\rz\times\Sigma)\hookrightarrow
 C^\infty\big(\rz\times\Sigma,\scrL(H^s_{p}(\rz^n_+),H^{s-d}_p(\rz^n_+))\big).
\end{equation}

\section{$\scrR$-boundedness of families from Boutet de Monvel's calculus}
\label{sec:rbounded_boutet}

Due to \eqref{eq:Sobmap}, operators of non-positive order and type zero induce families of continuous operators in
$L_p$-spaces. We are now going to analyze the $\scrR$-boundedness of these families.
First we consider strongly parameter-dependent Green operators. They can be treated using their particular symbol
kernel structure exhibited in Proposition \ref{prop:kernel}.

\begin{theorem}\label{thm:green1}
Let $d\le 0$ and {$1<p<\infty$}. Then
\begin{equation*}\label{eq:A}
 G^{d,0}(\rz^n_+;\rz\times\Sigma)
 \hookrightarrow
  S^{d}_\scrR(\rz\times\Sigma; L_p(\rz^n_+),L_p(\rz^n_+))
\end{equation*}
$($where the latter space is defined as in Corollary $\mathrm{\ref{cor:S-scrR}}$, replacing $\rz^\ell$ by
$\rz\times\Sigma)$.
\end{theorem}
\begin{proof}
For convenience we shall use the short-hand notations $G^{d,0}$, $G^{d,0}_{\mathrm{const}}$, and $S^d_\scrR$.

\textbf{Step 1.}
We first consider operators with symbol kernel independent of $x^\prime$. Let $G\in G^{d,0}_{\mathrm{const}}$
have symbol kernel $k$. Define $\mathfrak{g}(\xi^\prime,\tau,\mu):L_p(\rz_+)\to L_p(\rz_+)$ by
\begin{equation}\label{eq:symb}
 [\mathfrak{g}(\xi^\prime,\tau,\mu) u](x_n)=\int_0^\infty k(\xi^\prime,\tau,\mu;x_n,y_n)u(y_n)\,dy_n.
\end{equation}
Then $G(\tau,\mu)$ can be understood as the Fourier multiplier with symbol $\mathfrak{g}(\cdot,\tau,\mu)$.
In view of Theorem \ref{thm:girardi} it suffices to show that
\begin{equation*}
 \Big\{\spk{\tau,\mu}^{-d+k+|\gamma|}\spk{\xi^\prime}^{|\alpha|}
 D^\alpha_{\xi^\prime} D^k_\tau D^\gamma_\mu\mathfrak{g}(\xi^\prime,\tau,\mu)\st
 (\xi^\prime,\tau,\mu)\in\rz^{n-1}\times\rz\times\Sigma\Big\}
\end{equation*}
is an $\scrR$-bounded subset of $\scrL(L_p(\rz_+))$. Since
 $$\spk{\tau,\mu}^{-d+k+|\gamma|}\spk{\xi^\prime}^{|\alpha|}
   \le\spk{\xi^\prime,\tau,\mu}^{-d+|\alpha|+k+|\gamma|},$$
this follows with Kahane's contraction principle\footnote{This principle states that the inequality
 $$\sum_{z_1,\ldots,z_N=\pm1}\Big\|\sum_{j=1}^N z_j\alpha_jx_j\Big\|^q\le 2^q
     \sum_{z_1,\ldots,z_N=\pm1}\Big\|\sum_{j=1}^N z_j\beta_jx_j\Big\|^q$$
holds true whenever $\alpha_j,\beta_j\in\cz$ with $|\alpha_j|\le|\beta_j|$ and $x_1,\ldots,x_N\in X$ with arbitrary $N$.}
if we show the $\scrR$-boundedness of
\begin{equation}\label{eq:B2}
 \Big\{\spk{\xi^\prime,\tau,\mu}^{-d+|\alpha|+k+|\gamma|}D^\alpha_{\xi^\prime}
 D^k_\tau D^\gamma_\mu\mathfrak{g}(\xi^\prime,\tau,\mu)\st
 (\xi^\prime,\tau,\mu)\in\rz^{n-1}\times\rz\times\Sigma\Big\}.
\end{equation}
Since $\spk{\xi^\prime,\tau,\mu}^{-d+|\alpha|+k+|\gamma|}D^\alpha_\xi D^k_\tau D^\gamma_\mu \mathfrak{g}$
is a finite linear combination of symbols like $\mathfrak{g}$ we may assume without loss of generality that
$|\alpha|=k=|\gamma|=0$. Then we can estimate
\begin{align}\label{eq:12345}
\begin{split}
 \big|k(\xi^\prime,\tau,\mu;x_n,y_n)\big|
 &=\big|\wt{k}(\xi^\prime,\tau,\mu;\spk{\xi^\prime,\tau,\mu}x_n,
        \spk{\xi^\prime,\tau,\mu}y_n)\big|\\
 &\le C\,\spk{\xi^\prime,\tau,\mu}^{d+1}
        \big(\spk{\xi^\prime,\tau,\mu}(x_n+y_n)\big)^{-1}\\
 &\le C\,\spk{\tau,\mu}^d\frac{1}{x_n+y_n},
\end{split}
\end{align}
since $\wt{k}$ behaves like a symbol of order $d+1$ in $(\xi^\prime,\tau,\mu)$ and is rapidly decreasing in $(s_n,t_n)$.
Now the $\scrR$-boundedness of \eqref{eq:B2} follows from Theorem \ref{thm:intop}.

Since  $G^{d,0}_{\mathrm{const}}$ is continuously embedded in $\scrC(\rz\times\Sigma; \scrL(L_p(\rz^n_+))$,
{cf.\ \eqref{eq:smooth123}}, the closed graph theorem implies the continuity of the embedding into $ S^{d}_\scrR$.

\textbf{Step 2.} Due to Step 1, $G^{d,0}_{\mathrm{const}}\hookrightarrow S^d_\scrR$. In other words,
for any semi-norm $p(\cdot)$ of $S^d_\scrR$ there exists a semi-norm $q(\cdot)$ of
$G^{d,0}_{\mathrm{const}}$ such that $p(G)\le q(G)$ for any $G\in G^{d,0}_{\mathrm{const}}$.

For a function $f\in S^0_\cl:= S^0_\cl(\rz^{n-1}_{x^\prime})$ let $M_f$ denote the operator of multiplication,
$M_f\in \scrL(L_p(\rz^n_+))$. By \eqref{eq:rd} and \eqref{eq:tensor} we have the identification
$G^{d,0}=S^0_\cl\,\wh{\otimes}_\pi\,G^{d,0}_{\mathrm{const}}$. Now let
 $$G=\sum_{j=1}^N M_{f_j}G_j,\qquad f_j\in S^0_\cl,\quad G_j\in G^{d,0}_{\mathrm{const}}.$$
Then $G$ belongs to $S^d_\scrR$ and, with $p(\cdot)$ and $q(\cdot)$ as above,
 $$p(G)\le \sum_{j=1}^N p\big(M_{f_j}G_j\big)=\sum_{j=1}^N \|f_j\|_\infty p(G_j)
     \le\sum_{j=1}^N \|f_j\|_\infty q(G_j),$$
where $\|\cdot\|_\infty$ is the supremum-norm. By passing to the infimum over all possibilities to represent $G$
as such a linear combination we get
 $$p(G)\le \inf\Big\{\sum_{j=1}^N \|f_j\|_\infty q(G_j)\big)\stBig G=\sum_{j=1}^N M_{f_j}G_j\Big\}
     =:\wh{q}(G).$$
However, $\wh{q}(\cdot)$ induces a continuous semi-norm on the projective tensor product
$S^0_\cl\,\wh{\otimes}_\pi\,G^{d,0}_{\mathrm{const}}$, cf.\ the discussion after \eqref{eq:tensor}.
Since $p(\cdot)$ was arbitrary, we conclude that
$S^0_\cl\,\wh{\otimes}_\pi\,G^{d,0}_{\mathrm{const}}\hookrightarrow S^d_\scrR$.
\qed
\end{proof}

It seems that the direct proof of Theorem \ref{thm:green1} does not generalize to the case of weakly
parameter-dependent Green operators;
{
for example, estimate \eqref{eq:12345} in the weak case is only valid in case of regularity $\nu\ge1$. In fact,
for a function $f\in\scrS(\rz_+\times\rz_+)$, the identity
 $$|f(x_n,y_n)|^2=\int_{x_n}^\infty\int_{y_n}^\infty\partial_u\partial_v\big(f(u,v)\overline{f(u,v)}\big)\,dudv$$
gives the estimate
 $$|f(x_n,y_n)|^2\le 2\|f\|_{{L_2}}\|D_{x_n}D_{y_n} f\|_{L_2}+2\|D_{x_n}f\|_{L_2}\|D_{y_n} f\|_{L_2},$$
where the norm is that of $L_2(\rz_+\times\rz_+)$. Combining this with the estimates of \eqref{eq:kernel5},
for a symbol kernel $k\in R^{d,\nu}_\const(\rz^{n-1}\times\rz\times\Sigma)$, $1/2\le\nu<1$, we obtain only
 $$\big|k(\xi^\prime,\tau,\mu;x_n,y_n)\big|
     \le C\,\spk{\mu}^{\frac{\nu-1}{2}}\spk{\xi^\prime,\tau,\mu}^{d+\frac{1-\nu}{2}}\frac{1}{x_n+y_n}$$
which is weaker than the estimate \eqref{eq:12345}.
}

Thus we proceed differently, combining results of {Grubb, Kokholm}
\cite{grubb-kokholm93} on mapping properties of Green operators in weighted $L_2$-spaces
and the stability of $\scrR$-boundedness under interpolation. To this end, we shall make use of the spaces
\begin{equation*}
 L^\delta_2(\rz_+)={L_2}(\rz_+,t^{2\delta}dt),\qquad \delta\in\rz
\end{equation*}
and the following embeddings ({Theorem 1.9 of Grubb, Kokholm}  \cite{grubb-kokholm93}).

\begin{theorem}\label{thm:int}
Let $p\ge 2$ be given. Then, for any choice of $0<\delta^\prime<\frac{1}{2}-\frac{1}{p}<\delta<1$,
 $$\big(H^{\delta^\prime}_2(\rz_+),H^{\delta}_2(\rz_+)\big)_{\theta,p}
   \hookrightarrow
   L_p(\rz_+)
   \hookrightarrow
   \big(L_2^{-\delta^\prime}(\rz_+),L_2^{-\delta}(\rz_+)\big)_{\theta,p}$$
where $\theta$ is chosen such that
$\theta\delta+(1-\theta)\delta^\prime=\frac{1}{2}-\frac{1}{p}$ and $(\cdot,\cdot)_{\theta,p}$ refers to
real interpolation.
\end{theorem}

Moreover, let us introduce the Fr\'{e}chet space $S^d_{\scrR,w}(\rz\times\Sigma;X,Y)$ of smooth functions
$a:\rz\times\Sigma\to\scrL(X,Y)$ for which the sets
 $$T_{k,\gamma}(a):=\big\{\spk{\tau,\mu}^{-d}\spk{\tau}^{k}\spk{\mu}^{|\gamma|}
   D^k_\tau D^\gamma_\mu a(\tau,\mu)\st (\tau,\mu)\in\rz\times\Sigma\big\}$$
are $\scrR$-bounded for any choice of $k$ and $\gamma$. The semi-norms are defined as the $\scrR$-bounds
of the sets $T_{k,\gamma}$.

\begin{theorem}\label{thm:green2}
If $d\le 0$, $\nu\ge 1/2$ and {$1<p<\infty$} then
 $$G^{d,0;\nu}(\rz^n_+;\rz\times\Sigma)
     \hookrightarrow
     S^d_{\scrR,w}(\rz\times\Sigma;L_p(\rz^n_+),L_p(\rz^n_+)).$$
\end{theorem}
\begin{proof}
By Theorem \ref{thm:dual} and Theorem \ref{thm:adjoint} we may assume that $p\ge2$.
Using a tensor product argument as in the second step of the proof of Theorem \ref{thm:green1} reduces the proof
to showing that $G^{d,0;\nu}_\const \hookrightarrow S^d_{\scrR,w}$.

Thus let $G\in G^{d,0;\nu}_\const$.We have to show that
 $$\big\{\spk{\tau,\mu}^{-d}\spk{\tau}^{k}\spk{\mu}^{|\gamma|}
   D^k_\tau D^\gamma_\mu G(\tau,\mu)\st
   (\tau,\mu)\in\rz\times\Sigma\big\}$$
is an $\scrR$-bounded subset of $\scrL(L_p(\rz^n_+))$.
To this end represent $G$ as a Fourier multiplier with symbol $\mathfrak{g}(\xi^\prime,\tau,\mu)$ as done in the
proof of  Theorem \ref{thm:green1}. Due to Theorem \ref{thm:girardi} it suffices to show that
\begin{equation*}
 \Big\{\spk{\tau,\mu}^{-d}\spk{\tau}^{k}\spk{\mu}^{|\gamma|}
 \spk{\xi^\prime}^{|\alpha|}
 D^\alpha_{\xi^\prime}D^k_\tau D^\gamma_\mu
 \mathfrak{g}(\xi^\prime,\tau,\mu)
 \st (\xi^\prime,\tau,\mu)\in\rz^{n-1}\times\rz\times\Sigma\Big\}.
\end{equation*}
is an $\scrR$-bounded subset of $\scrL(L_p(\rz_+))$. Observing that
 $$\spk{\tau,\mu}^{-d}\spk{\tau}^{k}\spk{\mu}^{|\gamma|}
   \spk{\xi^\prime}^{|\alpha|}
   \le \spk{\xi^\prime,\tau,\mu}^{-d+|\gamma|}
   \spk{\xi^\prime,\tau}^{k+|\alpha|},$$
that $\spk{\xi^\prime,\tau,\mu}^{|\gamma|}D^\gamma_\mu\mathfrak{g}$ has the same structure as $\mathfrak{g}$,
and using Kahane's contraction principle, we may assume $\gamma=0$ and show that
\begin{equation*}
 M_{\alpha,k}:=
 \Big\{\spk{\xi^\prime,\tau,\mu}^{-d}\spk{\xi^\prime,\tau}^{k+|\alpha|}
 D^\alpha_{\xi^\prime}D^k_\tau \mathfrak{g}(\xi^\prime,\tau,\mu)
 \st (\xi^\prime,\tau,\mu)\in\rz^{n-1}\times\rz\times\Sigma\Big\}
\end{equation*}
is an $\scrR$-bounded subset of $\scrL(L_p(\rz_+))$. We know from Theorem 4.1.(5) of  {Grubb, Kokholm } \cite{grubb-kokholm93}
(see actually (4.15) in its proof) that for any $0<\eps<\frac{1}{2}$
\begin{equation*}
 M_{\alpha,k}\subset \scrL(L_2^{-\eps}(\rz_+),H^\eps_2(\rz_+))
\end{equation*}
is a bounded set. Since the involved spaces are Hilbert spaces, boundedness coincides with $\scrR$-boundedness.
Then using Theorem \ref{thm:int} $($with $\eps=\delta$ and $\eps=\delta^\prime$ where
$0<\delta^\prime<\frac{1}{2}-\frac{1}{p}<\delta<\frac{1}{2}$, respectively$)$ and Theorem \ref{thm:inter} we
obtain the $\scrR$-boundedness of $M_{\alpha,k}$ in $\scrL(L_p(\rz_+))$.

Since from {Grubb, Kokholm}  \cite{grubb-kokholm93} we know that the norm-bound of $M_{\alpha,k}$ can be estimated in terms of
semi-norms of $G$, an application of the closed graph theorem yields the continuity of the embedding.
\qed
\end{proof}

Finally, let us consider a family of pseudodifferential operators
 $$A_+(\tau,\mu)=\op_+(a)(\tau,\mu):
     L_p(\rz^n_+) \lra L_p(\rz^n_+)$$
with a symbol $a\in S^{d}(\rz^n\times\rz^n\times\rz\times\Sigma)$ with $d\le0$, cf.\  {\eqref{eq:Sd}}. Since we consider
the operator between $L_p$-spaces only $($and not between Sobolev spaces of higher regularity$)$ it is now not
necesssary to require the transmission property for $a$. We will show that
  $$A_+\in S^{d}_\scrR(\rz\times\Sigma; L_p(\rz^n_+),L_p(\rz^n_+)).$$
Since $\op_+(a)=r_+\op(a)e_+$ with the continuous operators $e_+:L_p(\rz^n_+)\to L_p(\rz^n)$ and
$r_+:L_p(\rz^n)\to L_p(\rz^n_+)$ of extension and restriction, respectively, it suffices to show that
  $$\op(a)\in S^{d}_\scrR(\rz\times\Sigma; L_p(\rz^n),L_p(\rz^n)).$$
Again by a tensor product argument analogous to that of Step 2 in the proof of Theorem \ref{thm:green1},
we can assume that $a$ has constant coefficients, i.e., $a\in S^{d}_\const$. However, then the statement follows
immediately from Theorem \ref{thm:girardi}, choosing there $X=Y=\cz$. Thus we can conclude:

\begin{theorem}\label{thm:main1}
Let $d\le 0$, $\nu\ge 1/2$, and $1<p<\infty$. Then
\begin{align*}
 B^{d,0}(\rz^n_+;\rz\times\Sigma)
 &\hookrightarrow
 S^{d}_\scrR(\rz\times\Sigma; L_p(\rz^n_+),L_p(\rz^n_+)),\\
 B^{d,0;\nu}(\rz^n_+;\rz\times\Sigma)
 &\hookrightarrow
 S^{d}_{\scrR,w}(\rz\times\Sigma; L_p(\rz^n_+),L_p(\rz^n_+)).
\end{align*}
Recalling the definition of the spaces $ S^{d}_{\scrR}$ and $ S^{d}_{\scrR,w}$ this means that if
$P(\tau,\mu)\in B^{d,0}(\rz^n_+;\rz\times\Sigma)$ and
$Q(\tau,\mu)\in B^{d,0;\nu}(\rz^n_+;\rz\times\Sigma)$ then
\begin{align*}
 &\big\{\spk{\tau,\mu}^{-d+k+|\gamma|} D^k_\tau D^\gamma_\mu P(\tau,\mu)\st
     (\tau,\mu)\in\rz\times\Sigma\big\},\\
 &\big\{\spk{\tau,\mu}^{-d}\spk{\tau}^{k}\spk{\mu}^{|\gamma|}D^k_\tau D^\gamma_\mu Q(\tau,\mu)\st
     (\tau,\mu)\in\rz\times\Sigma\big\}
\end{align*}
are $\scrR$-bounded subsets of $\scrL(L_p(\rz^n_+))$.
\end{theorem}

\begin{corollary}\label{cor:4.5}
Let $P(\tau,\mu)\in B^{d,0}(\rz^n_+;\rz\times\Sigma)$ or $P(\tau,\mu)\in B^{d,0,\nu}(\rz^n_+;\rz\times \Sigma)$
with $d\le 0$, $\nu\ge 1/2$, and $p\in (1,\infty)$. Define
$(\op_\tau P)(\mu) := \scrF^{-1}_{\tau\to r} P(\tau,\mu)\scrF_{r\to \tau}$. Then we have
  \[ \op_\tau P\in S_{\scrR}^d\big(\Sigma; L_p(\rz^{n+1}_+, L_p(\rz^{n+1}_+))\big).\]
\end{corollary}

\begin{proof}
{
Due to Theorem~\ref{thm:girardi}, we have to show, for $k=0,1$ and all $\gamma\in\nz_0^n$, the $\scrR$-boundedness
of the set
  \[ \big\{ \spk{\mu}^{-d+|\gamma|}D_\mu^\gamma \tau^k D_\tau^k P(\tau,\mu)\;|\; \tau\in\rz\setminus\{0\},
  \mu\in\Sigma\big\}.\]
In both cases, this follows from Kahane's inequality and Theorem~\ref{thm:main1}, since
  $\spk\mu^{-d+|\gamma|}\tau^k \le \spk{\tau,\mu}^{-d+k+|\gamma|}$ and
  $\spk{\tau,\mu}^{-d}\spk{\tau}^k\spk\mu^{|\gamma|} \le \spk{\tau,\mu}^{-d+k+|\gamma|}$. \qed
}
\end{proof}

In applications, the complex parameter $\mu$ is related to the spectral parameter $\lambda$ appearing in the
resolvent of the $L_p$-realization of a non-local boundary value problem. We included a second parameter
$\tau\in\rz$ in order to be able to treat additional parameters arizing from the problem itself, e.g., in the form of a
covariable in the unbounded direction of a waveguide. In this case, Corollary~$\mathrm{\ref{cor:4.5}}$
leads to maximal $L_p$-regularity by an application of the Theorem of Weis~\cite{weis01}.

\section{Maximal $L_p$-regularity for non-local {boundary value problems} in a wave-guide}\label{sec:stokes}

We will study non-local boundary value problems in a wave-guide, i.e., on a cylinder $\rz\times M$ whose
cross-section
is a smooth compact manifold $M$ with boundary $\partial M$. For this, we need to provide
some material on Boutet de Monvel's calculus on manifolds and the corresponding concept of parameter-ellipticity.
We follow {Grubb} \cite{grubb86} and {Grubb, Kokholm} \cite{grubb-kokholm93}. As an application, we study the reduced Stokes problem
in a waveguide in Section~\ref{sec:5.2}.

\subsection{Manifolds with boundary and parameter-ellipticity}\label{sec:manifolds}

In this section we indicate  how the calculus can be modified to cover domains with smooth boundary and how
it is used to describe solution operators of certain non-local boundary value problems. In the sequel we let $M$
denote a smooth compact manifold with boundary.  In view of the formulation of parameter-ellipticity given below,
we need to descibe a refined subclass of the class of Green operators introduced in Section \ref{sec:2.2} as well
as to introduce another type of operators, the so-called Poisson operators.

\subsubsection{Polyhomogeneous Green operators}\label{sec:5.1}

Let $G(\tau,\mu)$ be a weakly parameter-de\-pen\-dent Green operator of order $d$, type 0, and regularity $\nu$
as described in Definition \ref{def:green1}. We call $G(\tau,\mu)$ polyhomogeneous or classical if there exists a
sequence of Green operators $G_{d-j}(\tau,\mu)$, $j\in\nz_0$, such that,
for any $N\in\nz_0$,
 $$G(\tau,\mu)-\sum_{j=0}^{N-1} G_{d-j}(\tau,\mu)\;\in\; G^{d-N,0,\nu-N}(\rz^n_+;\rz\times\Sigma),$$
and if $k_{d-j}$ are the symbol kernels associated with $G_{d-j}$ as in \eqref{eq:kernel4} and \eqref{eq:kernel5}
$($with $d$ replaced by $d-j)$ it holds
\begin{equation}\label{eq:hom}
 k_{d-j}(x^\prime,t\xi^\prime,t\tau,t\mu;x_n/t,y_n/t)=t^{d-j}k_{d-j}(x^\prime,\xi^\prime,\tau,\mu;x_n,y_n)
\end{equation}
whenever $t\ge 1$ and ${|(\xi^\prime,\tau)|\ge 1}$. Extension by homogeneity allows us to associate
with $k_{d-j}$ a symbol kernel $k_{d-j}^h$ defined for ${(\xi^\prime,\tau)\not=0}$ and satisfying
\eqref{eq:hom} whenever $t>0$ and
${(\xi^\prime,\tau)\not=0}$. With this symbol kernel we associate an operator-valued function
$\mathfrak{g}^h_{d-j}(x^\prime,\xi^\prime,\tau,\mu)$, ${(\xi^\prime,\tau)\not=0}$, as in \eqref{eq:symb}.
The component of highest degree, $\mathfrak{g}^h_d$, is called the principal boundary symbol of $G$.

If $G$ is strongly para\-me\-ter-dependent, the previous definitions are slightly modified, asking the equality
in \eqref{eq:hom} to hold whenever $t\ge 1$ and $|(\xi^\prime,\tau,\mu)|\ge 1$. Then all
$\mathfrak{g}^h_{d-j}(x^\prime,\xi^\prime,\tau,\mu)$ are defined for $(\xi^\prime,\tau,\mu)\not=0$.
We denote the resulting classes by $G^{d,0,\nu}_{\cl}(\rz^n_+;\rz\times\Sigma)$ and
$G^{d,0}_{\cl}(\rz^n_+;\rz\times\Sigma)$, respectively.

Forming finite sums as in \eqref{eq:sum} yields operators of type $r\in\nz$.
In this case the principal boundary symbol is
given by
\begin{equation*}
 \mathfrak{g}^h_{d}(x^\prime,\xi^\prime,\tau,\mu)=\mathfrak{g}^h_{0,d}(x^\prime,\xi^\prime,\tau,\mu)
+\sum_{j=1}^{r} \mathfrak{g}^h_{j,d-j}(x^\prime,\xi^\prime,\tau,\mu)D_{x_n}^j.
\end{equation*}

\begin{definition}\label{def:globalgreen}
A weakly parameter-dependent negligible Green operator $C$ of type $r=0$ and regularity $\nu^\prime\in\rz$ on
$M$ is an integral-operator with kernel
 $$k(\tau,\mu;x,x^\prime)\;\in\;\scrC^\infty(\rz\times\Sigma\times M\times M)$$
$($smoothness up to the boundary$)$ that satisfies estimates
 $$p\big(D^k_\tau D^\alpha_\mu k(\tau,\mu;\cdot,\cdot)\big)\le
     C_{p\alpha k N}\spk{\mu}^{\frac{1}{2}-\nu^\prime-|\alpha|}\spk{\tau}^{-N}$$
for any continuous semi-norm $p$ of $\scrC^\infty(M\times M)$, all orders of derivatives and all $N\in\nz$.
In case of strong parameter-dependence we ask that $k$  is rapidly decreasing in $(\tau,\mu)$,
 $$k(\tau,\mu;x,x^\prime)\;\in\;\scrS\big(\rz\times\Sigma,\scrC^\infty(M\times M)\big).$$
Negligible operators of general type $r\in\nz$ are of the form
 $$C(\tau,\mu)=\sum_{j=0}^{r} C_{j}(\tau,\mu)D^j,$$
where the $C_j$ are negligible of type $0$ and regularity $\nu^\prime$ and $D$ denotes a first order differential
operator on $M$ which in a collar neighborhood of the boundary coincides with the derivative in normal direction.
\end{definition}

{Using a covering of $M$ with local coordinate systems and a subordinate partition of unity},
we can now define the classes of (global) parameter-dependent Green operators
$G^{d,r,\nu}_{\cl}(M;\rz\times\Sigma)$ and $G^{d,r}_{\cl}(M;\rz\times\Sigma)$, using the correponding classes
on the half-space and the
negligible operators of the previous definition, where in case of finite regularity $\nu$ the negligible remainders
are required to have regularity $\nu^\prime=\nu-d$.
With any such operator we can associate a principal boundary symbol, using the local principal boundary symbols,
which is defined on
$(T^*\partial M\setminus\{0\})\times\rz\times\Sigma$ in case of weak parameter-dependence and on
$(T^*\partial M\times\rz\times\Sigma)\setminus\{0\}$
in case of strong parameter-dependence. Here, $T^*\partial M$ denotes the cotangent bundle of $\partial M$.

\subsubsection{Poisson Operators}\label{sec:5.2.b}

Parameter-dependent Poisson operators on the half-space are of the form
\begin{equation*}
 [K(\tau,\mu)u](x)=\int e^{ix^\prime\xi^\prime}
 k(x^\prime,\xi^\prime,\tau,\mu;x_n)\wh{u}(\xi^\prime)\,\dbar\xi^\prime
\end{equation*}
where $u(x^\prime)$ is defined on the boundary of $\rz^n_+$ and the symbol kernel has a specific structure.
Poisson operators have an order $d$ and a regularity $\nu$, but there is no type {involved}.
The mentioned structure of a Poisson operator of order $d$ and finite or infinite regularity $\nu$ is obtained by
repeating all the constructions
of Section \ref{sec:2.2} concerning Green operators of type $r=0$ and regularity $\nu$ by simply eliminating the
$y_n$-variable and replacing $d$ by
$d-1/2$. Such a Poisson operator induces $($pointwise, for each $(\tau,\mu))$ continuous maps
 $$B^{s+d-1/p}_{pp}(\rz^{n-1})\lra H^{s}_p({\rz^n_+}),\qquad s\in\rz.$$
To obtain polyhomogeneous Poisson operators one needs to repeat the construction of the previous
Section \ref{sec:5.1}, again cancelling the $y_n$-variable.

Again these constructions can be generalized to the case of a manifold $M$, using
{local coordinate systems and} a partition of unity {as well as}
an analogue of Definition \ref{def:globalgreen}, replacing $\scrC^\infty(M\times M)$ by
$\scrC^\infty(M\times\partial M)$.
The resulting classes we shall denote by $P^{d,\nu}_{\cl}(M;\rz\times\Sigma)$ and $P^{d}_{\cl}(M;\rz\times\Sigma)$,
respectively.

\subsubsection{Parameter-elliptic boundary value problems}\label{sec:5.3}

Let
 $$A(\tau,\mu)=\sum_{j+k+\ell+|\alpha|\le 2}a_{jk\ell\alpha}(x^\prime,x_n)\tau^k\mu^\ell D^\alpha_{x^\prime}D^j_{x_n},
     \qquad (\tau,\mu)\in\rz\times\Sigma,$$
be a parameter-dependent differential operator on the half-space $\rz^n_+$ with coefficients that are
smooth up to the boundary. We associate with $A(\tau,\mu)$ two principal symbols, the usual homogeneous
principal symbol
 $$\sum_{j+k+\ell+|\alpha|= 2}a_{j\alpha}(x^\prime,x_n)\tau^k\mu^\ell{\xi^\prime}^\alpha_{x^\prime}\xi_n^j,
     \qquad (\xi,\tau,\mu)\not=0,$$
and the principal boundary symbol
 $$\mathfrak{a}^h_2(x^\prime,\xi^\prime,\tau,\mu)=
     \sum_{j+k+\ell+|\alpha|=2}a_{j\alpha}(x^\prime,0)\tau^k\mu^\ell{\xi^\prime}^\alpha_{x^\prime}D^j_{x_n},
     \qquad (\xi^\prime,\tau,\mu)\not=0.$$
We call $A$ $($interior$)$ parameter-elliptic if its usual homogeneous principal symbol is pointwise invertible.
Similar constructions make sense on
the manifold $M$, leading to principal symbols on $(T^*M\times\rz\times\Sigma)\setminus\{0\}$ and on
$(T^*\partial M\times\rz\times\Sigma)\setminus\{0\}$, respectively.

The following theorem is a very special version of results due to Grubb. We have chosen to only state this
special version, since it suffices for our application to the reduced Stokes problem in the next section and
since in this way we can keep the exposition shorter. In fact, one may admit more general classes of
pseudodifferential operators $A(\tau,\mu)$ of arbitrary order acting between vector bundles as well
as more general boundary conditions. For details we refer to {Grubb} \cite{grubb86} and {Grubb, Kokholm} \cite{grubb-kokholm93}.

\begin{theorem}\label{thm:grubb}
Let $A(\tau,\mu)$ be a second order parameter-dependent differential operator on $M$,
$G(\tau,\mu)\in G^{2,r,\nu}_{\cl}(M;\rz\times\Sigma)$ a weakly parameter-dependent polyhomogeneous
Green operator of type $r\le 2$ and regularity $\nu\ge 1/2$. Let $\gamma_0$ and $\gamma_1$ denote
Dirichlet and Neumann boundary conditions on $M$, respectively. The boundary value problem
\begin{align}\label{eq:bvp}
   \begin{pmatrix}A(\tau,\mu)+G(\tau,\mu) \\ \gamma_j\end{pmatrix}:\;
    H^{s}_p(M)\lra
   \begin{matrix} H^{s-2}_p(M)\\ \oplus \\ B^{s-j-1/p}_{pp}(\partial M)\end{matrix},\qquad s>1+1/p,
\end{align}
is called parameter-elliptic if $A(\tau,\mu)$ is interior parameter-elliptic and, whenever $\xi^\prime\not=0$,
the initial value problem
\begin{align*}
\begin{split}
   (\mathfrak{a}^h_2(x^\prime,\xi^\prime,\tau,\mu)+\mathfrak{g}^h_2(x^\prime,\xi^\prime,\tau,\mu))u
  &=0\qquad \text{on }\rz_+,\\
  (1-j)u(0)+ju^\prime (0)&=0
\end{split}
\end{align*}
has only the trivial solution $u=0$ in $\scrS(\rz_+)$.
In this case \eqref{eq:bvp} is an isomorphism for $|(\tau,\mu)|$
sufficiently large, and
\begin{align*}
   \begin{pmatrix}A(\tau,\mu)+G(\tau,\mu) \\ \gamma_j\end{pmatrix}^{-1}=
   \begin{pmatrix}P_j(\tau,\mu) & K_j(\tau,\mu)\end{pmatrix},
\end{align*}
with an operator $P_j(\tau,\mu)\in B^{-2,0,\nu}(M;\rz\times\Sigma)$ as described in Section
$\mathrm{\ref{sec:2.3}}$\footnote{The operator class on $M$ instead of the half-space is again obtained by
using a {covering by local coordinate systems and a subordinate}
partition of unity and taking into account the global smoothing remainders defined in
Definition \ref{def:globalgreen}.} and a Poisson operator $K_j(\tau,\mu)\in P^{-j,\nu}_{\cl}(M;\rz\times\Sigma)$.
\end{theorem}

\begin{corollary}
 \label{cor:5.3}
  In the situation of Theorem~$\mathrm{\ref{thm:grubb}}$, assume that $A(\tau,\mu) = \mu^2+\widetilde A(\tau)$
and that $G(\tau,\mu)=G(\tau)$ is independent of $\mu$. Let {$1<p<\infty$ and $T>0$ be finite}.
Define the operator $\mathbf A$ in $L_p(Z)$ with $Z:=\rz\times M$ by
  \begin{align*}
   \scrD(\mathbf A) & := \big\{ u\in W_p^2(Z)\,|\, \gamma_ju=0\;\text{on }\partial Z\big\},\\
   \mathbf A u &:= \op_\tau \widetilde A(\tau)u + \op_\tau G(\tau)u\quad (u\in\scrD(\mathbf A)).
   \end{align*}
   If the boundary value problem \eqref{eq:bvp} is parameter-elliptic in the sector
$\Sigma:= \{\mu\in\cz\setminus\{0\}: |\arg \mu|\le \frac\pi 4\}\cup\{0\}$, then $\mathbf A$ has maximal
$L_q$-regularity for every $q\in (1,\infty)$, i.e., the mapping
   \[ \partial_t+\mathbf A\colon
W_q^1\big((0,T);L_p(Z)\big)\cap L_q\big((0,T);\scrD(\mathbf A)\big)\to L_q\big((0,T); L_p(Z)\big)\]
   is an isomorphism of Banach spaces.
\end{corollary}

\begin{proof}
Define the operator $A_M(\tau)$ for $\tau\in\rz$ by $\scrD(A_M(\tau)) := \{ v\in W_p^2(M): \gamma_jv=0\}$
and $A_M(\tau)v := \widetilde A(\tau) + G(\tau)$. By Theorem~$\mathrm{\ref{thm:grubb}}$, the resolvent
$(\mu^2+A_M(\tau))^{-1}$ exists for sufficiently large $\mu\in\Sigma$ and is given by
$P_j(\tau,\mu)\in B^{-2,0,\nu}(M;\rz\times\Sigma)$. Choosing $\lambda_0>0$ sufficiently large, we obtain
\[ \mu^2 \big(\mu^2+\lambda_0+A_M(\tau)\big)^{-1}\in B^{0,0,\nu}(M;\rz\times \Sigma).\]
Setting $\lambda=\mu^2$, Corollary~\ref{cor:4.5} yields
\[ \lambda\big(\lambda+\lambda_0+\mathbf A)^{-1} = \op_\tau\big[ \mu^2(\mu^2
+\lambda_0+A_M(\tau))^{-1}\big] \in S_{\scrR}^0(\Sigma;L_p(Z),L_p(Z)).\]
By the Theorem of Weis~\cite{weis01}, $\mathbf A+\lambda_0$ has maximal $L_q$-regularity for all $1<q<\infty$.
As the time interval $(0,T)$ is assumed to be finite, this gives maximal $L_q$-regularity for $\mathbf A$.
\qed
\end{proof}

As indicated at the end of Section \ref{sec:rbounded_boutet}, the analog results hold for
$\tau$-independent operators, i.e., for parameter-elliptic boundary value problems of the form
  \begin{align*}
    (\lambda+A+G) u & = f \quad\text{in }M,\\
    \gamma_ju&=0 \quad\text{on }\partial M,
  \end{align*}
  where $A$ and $G$ are $($parameter-independent$)$ pseudodifferential and Green operators, respectively.

\subsection{The reduced Stokes problem}\label{sec:5.2}

Let $\Sigma := \{\mu\in \cz\setminus\{0\}: |\arg\mu|\le \theta\}\cup\{0\}$ with $\frac\pi4<\theta<\frac\pi2$.
For $\mu\in\Sigma$, we consider in the waveguide $Z:=\rz\times M$ the resolvent problem
\begin{align}\label{eq:eqn1}
\begin{split}
  \mu^2 u -\Delta u+\nabla p&=f \qquad \text{in }Z,  \\
  \mathrm{div}\,u&=0 \qquad  \text{in }Z,  \\
  \gamma_0u&=0 \qquad  \text{on }\partial Z,
\end{split}
\end{align}
where $\gamma_0$ denotes the operator of restriction to the boundary.
{
We write the Laplacian $\Delta$ on $Z$ and the inner normal $\nu$ of $Z$ as $\Delta=\partial_{r}^2+\Delta_M$
and $\nu=(0,\nu_M)$, respectively, where $r$ denotes the variable of $\rz$ and the subscript $M$ indicates the
corresponding objects on $M$. We define the boundary operators $\gamma_\nu$ and $\gamma_1$ on $Z$ by
$\gamma_\nu u=\nu\cdot\gamma_0 u$ and $\gamma_1 p=\gamma_\nu(\nabla p)$, respectively.
Moreover,  let us write $u=(u_1,\ulu)$ with $u_1:Z\to\rz$ and $\ulu:Z\to\rz^n$ and analogously $f=(f_1,\ulf)$.

Due to the divergence condition, in \eqref{eq:eqn1} we may replace the Laplacian $\Delta$ by
$A=\Delta-\nabla \mathrm{div}$ without changing the problem $($a `trick' going back to Grubb, Solonnikov
\cite{grubb-solonnikov91}), eliminating the second order derivatives in the direction normal to the boundary$)$.
Doing so, we obtain from \eqref{eq:eqn1} that
}
\begin{equation}\label{eq:eqn2}
\begin{aligned}
 \Delta p& =0&& \text{in }Z,\\
 \gamma_1 p& =\gamma_\nu Au+\gamma_\nu f&&\text{on }\partial Z,
\end{aligned}
\end{equation}
{for any $f$ with $\mathrm{div}\, f=0$}.
If $\mathbf{K}$ denotes the operator satisfying $\Delta\mathbf{K}=0$ and $\gamma_1\mathbf{K}=\mathrm{id}$,
the first equation in \eqref{eq:eqn1} becomes
\begin{equation}\label{eq:eqn1a}
 \mu^2 u -\Delta u+\nabla\mathbf{K}(\gamma_\nu Au+\gamma_\nu f)=f.
\end{equation}
$\mathbf{K}$ is the Fourier multiplier with symbol $K(\tau)$ satisfying
$(\Delta_M-\tau^2)K(\tau)=0$ and $\gamma_{1,M}K(\tau)=\mathrm{id}$ $($denoting the co-variable to $r$ by $\tau)$ .
{$K(\tau)$ does not belong to Boutet de Monvel's calculus, but we can say the following}:
\begin{lemma}\label{lem:abc}
There exists a $($strongly$)$ parameter-dependent Poisson operator $K_0(\tau)\in P_\cl^{-1}(M;\rz)$ and
an $\eps=\eps(p)>0$ such that
  $$\nabla (\mathbf{K}-\mathbf{K}_0)\gamma_\nu A:H^{2-\eps}_p(Z)^n\cap \mathrm{ker}\,\gamma_0 \lra L_p(Z)$$
continuously, where $\mathbf{K}$ and $\mathbf{K}_0$ denote the Fourier multipliers with symbol $K(\tau)$
and $K_0(\tau)$, respectively.
\end{lemma}
In fact, {using} a partition of unity and local coordinates, the proof of this lemma can be reduced to the model
case of
$M=\rz^n_+$  being the half space. In this case, the symbol kernels of $K(\tau)$ and $K_0(\tau)$ are
 $$k(\xi^\prime,\tau;x_n)=-\frac{1}{|(\xi^\prime,\tau)|}e^{-|(\xi^\prime,\tau)|x_n},\qquad
     k_0(\xi^\prime,\tau;x_n)=-\frac{1}{[(\xi^\prime,\tau)]}e^{-[(\xi^\prime,\tau)]x_n},$$
respectively, where $[\,\cdot\,]$ denotes a smooth function that coincides with the usual
modulus $|\cdot|$ {outside some ball}. We shall not go into further details here, but refer the reader to {Abels} \cite{abels05a}
for an analogous construction, in particular to Lemma 4.2 of that paper.

Now, instead of \eqref{eq:eqn1a}, we first consider the problem obtained by replacing $\mathbf{K}$ by $\mathbf{K}_0$, i.e.
\begin{align}\label{eq:eqn1b}
\begin{split}
  \mu^2 u -\Delta u+\nabla\mathbf{K}_0\gamma_\nu Au&=f^\prime \quad \ \ \text{in }Z,  \\
  \gamma_0u&=0 \qquad \text{on }\partial Z,
\end{split}
\end{align}
where $f^\prime=f-\nabla\mathbf{K}_0\gamma_\nu f$. The original problem \eqref{eq:eqn1} shall be treated below with
help of a suitable perturbation argument.

By direct calculation one sees that
 $$\gamma_\nu Au
     ={-\gamma_{1,M}\partial_r u_1}+\gamma_{\nu_M}A_M\ulu,\qquad A_M
     =\Delta_M-\nabla_M \mathrm{div}_M.$$
By writing $\nabla= (\partial_r,\nabla_M)$ and passing to the Fourier transform in $r$, we derive from
\eqref{eq:eqn1b} that
\begin{align}\label{eq:eqn1c}
\begin{split}
 (\mu^2&+\tau^2-\Delta_M)(U_1,\ulU)+\\
 &+(i\tau,\nabla_M)K_0(\tau)
     \big({-\gamma_{1,M}i\tau U_1}+\gamma_{\nu_M}A_M\ulU\big)
     =(F^\prime_1,\ulF^\prime)
\end{split}
\end{align}
with boundary conditions $\gamma_{0,M}U_1=0$ and $\gamma_{0,M}\ulU=0$.
{Here, capital letters indicate the Fourier transform of the respective function in the first variable}.
In particular, the equation {for the first component} can be written as
 $$C(\tau,\mu)U_1=B(\tau)\ulU+F^\prime_1,\qquad \gamma_{0,M}U_1=0,$$
where
\begin{align*}
 B(\tau)&={-i\tau K_0(\tau)\gamma_{\nu_M}A_M,}\\
 C(\tau,\mu)&=\mu^2+\tau^2-\Delta_M{+\tau^2K_0(\tau)\gamma_{1,M}.}
\end{align*}
Now, $B(\tau)$ is an operator of Boutet's calculus of order and type 2 with strong parameter-dependence on $\tau$,
while $C(\tau,\mu)$ has order and type 2 as well and is weakly parameter-dependent with regularity $\nu=1/2$; for
the latter see (2.3.55) in Proposition 2.3.14 of \cite{grubb86}.
By parameter-ellipticity and Theorem \ref{thm:grubb}
we can find
 $$\begin{pmatrix}C(\tau,\mu)\\ \gamma_{0,M}\end{pmatrix}^{-1}
     =:\begin{pmatrix}D(\tau,\mu) & \wt{K}(\tau,\mu)\end{pmatrix}$$
with $D$ {being} of order $-2$, $\wt{K}$ of order 0, and both having type 0 and regularity $\nu=1/2$. Therefore
\begin{equation}\label{eqn:milestone0}
 U_1=E(\tau,\mu)\ulU+D(\tau,\mu)F^\prime_1,\qquad E(\tau,\mu):=D(\tau,\mu)B(\tau);
\end{equation}
note that $E(\tau,\mu)$ is weakly parameter-dependent of zero order, type 2, and regularity $\nu=1/2$.
Inserting this in the equation for the second component in \eqref{eq:eqn1b}, we find the equation
\begin{equation}\label{eqn:milestone}
 (\mu^2+\tau^2-\Delta_M)\ulU+G(\tau,\mu)\ulU=\ulF^\prime-D(\tau,\mu)F^\prime_1
\end{equation}
for $\ulU$, where
 $$G(\tau,\mu)=\nabla_M K_0(\tau)\gamma_{\nu,M}A_M{-i\tau\nabla_MK_0(\tau)\gamma_{1,M}E(\tau,\mu)}$$
is a weakly parameter-dependent singular Green operator of order and type 2, with regularity $\nu=1/2$. Using the
parameter-ellipticity of $(\mu^2+\tau^2-\Delta_M+G(\tau,\mu),\,\gamma_{0,M})$ we can resolve
\eqref{eqn:milestone} for $\ulU$ and substitute $\ulU$ in \eqref{eqn:milestone0}, resulting in
  $$U(\tau,\cdot)=S(\tau,\mu)F^\prime(\tau,\cdot),\qquad |(\tau,\mu)|\ge R,$$
for some sufficiently large $R\ge0$ and with $S(\tau,\mu)$ being an $(n+1)\times(n+1)$-matrix with components
belonging to $B^{-2,0,1/2}(M;\rz\times\Sigma)$. Passing to the inverse Fourier transform
with respect to $\tau$, we see that $u=S(\mu)f^\prime$ is the unique solution of \eqref{eq:eqn1b} for $\mu\in\Sigma$
sufficiently large, where the solution operator $R(\mu)$ is defined by
$R(\mu) := \scrF^{-1}_{\tau\to r}S(\tau,\mu)\scrF_{r\to\tau}$.

Due to Theorem~\ref{thm:main1} we have, for sufficiently large $\mu_0>0$,
\begin{equation}
  \label{eq:5-14}
  \scrR\Big(\big\{ \spk\mu^{2+|\alpha|} \partial_\mu^\alpha R(\mu): \mu\in\Sigma, |\mu|\ge \mu_0\big\}\Big) <\infty.
\end{equation}
Therefore, for sufficiently large $\mu\in\Sigma$, the problem~\eqref{eq:eqn1b} is uniquely solvable,
and \eqref{eq:5-14} gives a resolvent estimate even in the $\scrR$-bounded version.
{This is the main step for proving maximal regularity for the Stokes operator:}

\begin{theorem}\label{thm:5.5}
Let $M\subset\rz^n$ be a bounded smooth domain, and $Z:=\rz\times M$.
{Let $1<p<\infty$ and $T>0$ be finite}.
Let $P_p$ be the Helmholtz projection in $L_p(Z)$ $($see~{Farwig} \emph{\cite{farwig03}}$)$.
Define the Stokes operator
$\mathbf A$ by
  \begin{align*}
    \scrD(\mathbf A) & := W_p^2(Z)\cap W_{p,0}^1(Z)\cap L_{p,\sigma}(Z),\\
    \mathbf A u & := -P_p\Delta u\quad (u\in\scrD(\mathbf A)).
  \end{align*}
  Here, $L_{p,\sigma}(Z)$ stands for the standard space of solenoidal $L_p$-vector fields.
Then $\mathbf A$ has maximal $L_q$-regularity for every {$1<q<\infty$} in the time interval $(0,T)$.
\end{theorem}

\begin{proof}
  Due to {the existence of} the Helmholtz decomposition of $L_p(Z)$ (see {Farwig} \cite{farwig03}), we see that for $f\in L_{p,\sigma}(Z)$ the
solvability of ${(\lambda+\mathbf A)}u=f$ in $L_{p,\sigma}(Z)$ is equivalent to the solvability of \eqref{eq:eqn1} with
$\mu^2=\lambda$. { Instead of \eqref{eq:eqn1}, we can consider the reduced Stokes problem~\eqref{eq:eqn1a} with Dirichlet boundary conditions. Therefore, we have to show maximal regularity (in finite time intervals) for the reduced Stokes operator $\mathbf A^{(r)}\colon D(\mathbf A^{(r)})\subset L_p(Z)\to L_p(Z)$ which is defined by $\scrD(\mathbf A^{(r)}) := W_p^2(Z)\cap W_{p,0}^1(Z)$ and
\[ \mathbf A^{(r)} u := -\Delta u + \nabla \mathbf K\gamma_\nu A u, \quad u\in\scrD (\mathbf A^{(r)}).\]
{Substituting $\mathbf K$ by $\mathbf K_0$ yields the modified reduced Stokes operator $\mathbf A_0^{(r)}$,}
\[ \mathbf A_0^{(r)} u := -\Delta u + \nabla \mathbf K_0\gamma_\nu A u, \quad u\in\scrD (\mathbf A_0^{(r)}){:=\scrD (\mathbf A^{(r)})}.\]
We have seen that, for sufficiently large $\lambda=\mu^2$ with $\mu\in\Sigma$, the operator $\lambda+\mathbf A_0^{(r)}$ is invertible, and that its inverse is given by  $(\lambda+\mathbf A_0^{(r)})^{-1} = R(\mu)$. Due to the $\scrR$-boundedness result in \eqref{eq:5-14} and the condition $\theta>\frac\pi 4$, the operator family
\[ \big\{ \lambda (\lambda+\lambda_0+\mathbf A_0^{(r)})^{-1}: \operatorname{\textrm{Re}}\, \lambda \ge 0\big\} \subset \scrL(L_p(Z))\]
is $\scrR$-bounded for sufficiently large $\lambda_0>0$.

The reduced Stokes operator $\mathbf A^{(r)}$ can be seen as a small perturbation of $\mathbf A_0^{(r)}$. In fact, due to Lemma~\ref{lem:abc}, for the difference $\mathbf B := \mathbf A_0^{(r)} - \mathbf A^{(r)}$ we have $\scrD(\mathbf B)\supset \scrD(\mathbf A_0^{(r)})$, and for every $\delta>0$ there exists a $C_\delta>0$ such that
\[ \| \mathbf B u \|_{L_p(Z)} \le C \| u\|_{H_p^{2-\epsilon}(Z)} \le \delta \| \mathbf A^{(r)}_0 u\|_{L_p(Z)} + C_\delta \|u\|_{L_p(Z)},\quad u\in \scrD(\mathbf A_0^{(r)}).\]
Here we used the interpolation inequality. Now, the perturbation result in Denk, Hieber, Pr\"uss \cite{denk-hieber-pruess03}, Proposition~4.2, yields the $\scrR$-boundedness of
\[ \big\{ \lambda (\lambda+\lambda_1+\mathbf A^{(r)})^{-1}: \operatorname{\textrm{Re}}\, \lambda \ge 0\big\} \subset \scrL(L_p(Z))\]
for some sufficiently large $\lambda_1>0$. By the Theorem of Weis, $\mathbf A^{(r)}+\lambda_1$ has maximal $L_q$-regularity for all $1<q<\infty$. As the time interval is finite, this gives maximal $L_q$-regularity for the reduced Stokes operator which finishes the proof.}
\qed
\end{proof}

Let us finally remark that more general results on maximal regularity for the Stokes operator in cylindrical domains
have been obtained, e.g., by Farwig and Ri~\cite{farwig-ri08} under weaker smoothness assumptions on the
domain. For the existence of a bounded $H^\infty$-calculus $($which also implies the statement of
Theorem~$\mathrm{\ref{thm:5.5}})$, we also refer to Abels~\cite{abels05a}.
{The intention of the present section was not to recover or even improve these results, but to outline that
maximal regularity can also be obtained by employing} the $\scrR$-boundedness of operator-families
belonging to Boutet de Monvel's calculus.


\bibliographystyle{spmpsci}

\end{document}